\documentclass[12pt]{amsart}
\usepackage[osf,sc]{mathpazo}
\usepackage{amssymb}

\usepackage{geometry}
\usepackage{bm}
\usepackage{graphicx}
\usepackage{tikz-cd}
\usepackage{hyperref}
\hypersetup{
    colorlinks=true, 
    linktoc=all,     
    linkcolor=blue,
        citecolor=red,
    filecolor=black,
    urlcolor=blue	
}
\usepackage{enumerate}
\usepackage[inline]{enumitem}
\makeatletter
\newcommand{\inlineitem}[1][]{%
\ifnum\enit@type=\tw@
    {\descriptionlabel{#1}}
  \hspace{\labelsep}%
\else
  \ifnum\enit@type=\z@
       \refstepcounter{\@listctr}\fi
    \quad\@itemlabel\hspace{\labelsep}%
\fi} \makeatother
\newcommand{\ga}{\alpha}
\newcommand{\gb}{\beta}
\newcommand{\gga}{\gamma}
\newcommand{\gd}{\delta}

\newcommand{\gl}{\lambda}

\newcommand{\gp}{\pi}

\newcommand{\gs}{\sigma}

\newcommand{\Gd}{\Delta}

\newcommand{\Gs}{\Sigma}

\newcommand{\Gf}{\Phi}

\newcommand{\Gom}{\Omega}

\newcommand{\subs}{\subset}

\newcommand{\bs}{\backslash}

\newcommand{\nin}{\notin}

\newcommand{\mbb}{\mathbb}

\newcommand{\mcl}{\mathcal}

\newcommand{\ul}{\underline}
\newcommand{\ol}{\overline}

\newcommand{\us}{\underset}
\newcommand{\os}{\overset}

\newcommand{\lra}{\longrightarrow}

\newcommand{\Z}{\mbb Z}

\newcommand{\ZZ}[1]{\Z/p^{#1}\Z}

\newcommand{\Ra}{\Rightarrow}
\newcommand{\Llra}{\Longleftrightarrow}

\newcommand{\equ}[1]{%
\begin{equation*}
#1
\end{equation*}
}
\newcommand{\equa}[1]{%
\begin{equation*}
\begin{aligned}
#1
\end{aligned}
\end{equation*}
}
\newcommand{\equan}[2]{%
\begin{equation}
\label{Eq:#1}
\begin{aligned}
#2
\end{aligned}
\end{equation}
}

\newcommand{\mattwo}[4]{%
\begin{pmatrix}
  #1 & #2\\ #3 & #4
\end{pmatrix}
}

\newcommand{\matcoltwo}[2]{%
\begin{pmatrix}
  #1\\#2
\end{pmatrix}
}

\newcommand{\matfour}[9]{%
  \def\argi{{#1}}%
  \def\argii{{#2}}%
  \def\argiii{{#3}}%
  \def\argiv{{#4}}%
  \def\argv{{#5}}%
  \def\argvi{{#6}}%
  \def\argvii{{#7}}%
  \def\argviii{{#8}}%
  \def\argix{{#9}}%
  \matfourRelay
}
\newcommand\matfourRelay[7]{%
\begin{pmatrix}
  \argi & \argii & \argiii & \argiv\\
  \argv & \argvi & \argvii & \argviii\\
  \argix & #1 & #2 & #3\\
   #4 & #5 & #6 & #7
\end{pmatrix}
}

\newtheorem{thm}{Theorem}[section]

\newtheorem{prop}[thm]{Proposition}

\newtheorem{cor}[thm]{Corollary}

\newtheorem{ques}[thm]{Question}

\newtheorem{example}[thm]{Example}

\makeatletter
\def\namedlabel#1#2{\begingroup
   \def\@currentlabel{#2}%
   \label{#1}\endgroup
}
\makeatother

\newtheorem*{thmOmega}{\bf{Theorem} $\bm{\Gom}$}
\newtheorem*{thmSigma}{\bf{Theorem} $\bm{\Gs}$}
\theoremstyle{definition}
\newtheorem{definition}[thm]{Definition}

\theoremstyle{remark}
\newtheorem{remark}[thm]{Remark}
\numberwithin{equation}{section}

\begin{document}
\title{On the Endomorphism Semigroups of Extra-special \MakeLowercase{p}-groups and Automorphism Orbits}
\author{C P Anil Kumar}
\address{School of Mathematics, Harish-Chandra Research Institute, HBNI, Chhatnag Road, Jhunsi, Prayagraj (Allahabad), 211 019,  India. \,\, email: {\tt akcp1728@gmail.com}}
\author{Soham Swadhin Pradhan}
\address{School of Mathematics, Harish-Chandra Research Institute, HBNI, Chhatnag Road, Jhunsi, Prayagraj (Allahabad), 211 019,  India. \,\, email: {\tt soham.spradhan@gmail.com}}
\date{\sc \today}
\begin{abstract}
For an odd prime $p$ and a positive integer $n$, it is well known that there are two types of extra-special $p$-groups of order $p^{2n+1}$, first one is the Heisenberg group which has exponent $p$ and the second one is of exponent $p^2$. In this article, a new way of representing the extra-special $p$-group of exponent $p^2$ is given. These representations facilitate an explicit way of finding formulae for any endomorphism and any automorphism of an extra-special $p$-group $G$ for both the types. Based on these formulae, the endomorphism semigroup $End(G)$ and the automorphism group $Aut(G)$ are described. The endomorphism semigroup image of any element in $G$ is found and the orbits under the action of the automorphism group $Aut(G)$ are determined.  As a consequence it is deduced that, under the notion of degeneration of elements in $G$, the endomorphism semigroup $End(G)$ induces a partial order on the automorphism orbits when $G$ is the Heisenberg group and does not induce when $G$ is the extra-special $p$-group of exponent $p^2$.  Finally we prove that the cardinality of isotropic subspaces of any fixed dimension in a non-degenerate symplectic space is a polynomial in $p$ with non-negative integer coefficients. Using this fact we compute the cardinality of $End(G)$.
\end{abstract}
\maketitle
\section{\bf{Introduction}}
\subsection{Preamble}
In the literature, for a prime $p$, a {\it special} group is defined as an elementary abelian $p$-group or a $p$-group where the Frattini subgroup, the commutator subgroup and the center coincide and the center is of exponent $p$. An {\it extra-special} $p$-group is a non-abelian special group where the center is of order $p$. The extra-special $p$-groups arise in various contexts and are well studied groups.

We mention three contexts. Firstly they occupy a distinctive place in the representation theory (D.~E.~Gorenstein~\cite{MR0231903} (Chapter $5$, Section $5$, Theorem $5.4$), L.~Dornhoff~\cite{MR0347959} (Chapter $31$, Theorem $31.5$), H.~Opolka~\cite{MR0486098}) and the cohomology of finite groups (D.~J.~Benson and J.~F.~Carlson~\cite{MR1157256},\cite{MR1233415}). Secondly the extra-special $p$-groups has generated considerable interest in the study of its non-commuting subsets from a group theoretic and combinatorial view point (A.~Y.~M.~Chin~\cite{MR2126728}, M.~Isaacs~\cite{MR687893}, H.~Liu and Y.~Wang~\cite{MR3058246},~\cite{MR3403691}). Thirdly, the automorphism group of an extra-special $p$-group is also an important aspect of study in the literature.
D.~L.~Winter~\cite{MR0297859} has determined the structure of $Aut(G)$ for an extra-special $p$-group $G$. More precisely he has proved that the automorphism group $Aut(G)$ is the semi-direct product of the normal subgroup $N$ of centrally trivial automorphisms, (that is, those automorphisms which act trivially on the center $\mcl{Z}(G)$) and a cyclic group of order $(p-1)$ generated by an automorphism of $G$ which is an extension of the generator of $Aut(\mcl{Z}(G))$. Moreover it is shown that the quotient group $\frac{N}{Inn(G)}$ of $N$ by the inner automorphism group $Inn(G)$ is isomorphic to a subgroup of a symplectic group whose structure is also known. It is also known that for an odd prime $p$, the group $Aut(G)$ is a split extension of the outer automorphism group $Out(G)$ by $Inn(G)$. For $p=2$, this need not be true as shown by R.~L.~Griess Jr.~\cite{MR0476878}. H.~Liu and Y.~Wang~\cite{MR2606849} have determined the structure of the automorphism group of a generalized extra-special $p$-group. 

In this article, for an odd prime $p$ and a positive integer $n$, we compute and give an explicit expression for an endomorphism and an automorphism of an extra-special $p$-group of order $p^{2n+1}$. More precisely, first we present in an explicitly new way, the extra-special $p$-group of order $p^{2n+1}$ and of exponent $p^2$ (Definition~\ref{defn:ESPII}), just similar to one of the standard  representations of the Heisenberg group of order $p^{2n+1}$ (Definition~\ref{defn:ESPI}). These definitions are advantageous to write down formulae for any endomorphism and any automorphism for both the types of groups (in main Theorems~\ref{theorem:ExtraSpecialTypeI},~\ref{theorem:ExtraSpecialTypeII}). In spite of the already determined structure of the automorphism group in the literature~\cite{MR0297859}, the formulae for endomorphisms and automorphisms given in this article can be derived in a very natural and elegant manner. The importance of these explicit formulae is that they facilitate us to compute the endomorphism semigroup images of elements in the group and the automorphism orbits. These are later used to explore the existence of partial order on automorphism orbits using the notion of {\it degeneration of elements} (Definition~\ref{defn:PODegeneration}). Similar work has been done for the case of finite abelian $p$-groups by K.~Dutta and A.~Prasad~\cite{MR2793603}. We have computed the cardinality of the automorphism group and the cardinality of the endomorphism semigroup of an extra-special $p$-group for both the types as a polynomial in $p$ with integer coefficients. While computing the cardinality of the endomorphism group we prove that the cardinality of isotropic subspaces of any fixed dimension in a non-degenerate symplectic space is a polynomial in $p$ with non-negative integer coefficients.  
\subsection{Statement of Main Theorems}
We begin this section with a few required definitions in order to state the main theorems.

\begin{definition}[Extra-special $p$-group of First Type: Heisenberg Group]
\label{defn:ESPI}
Let $p$ be an odd prime, $n$ be a positive integer and $\mbb{F}_p$ be the finite field order $p$. 
For $\ul{u}=(u_1,u_2,\ldots,u_n)^t,$ $\ul{w}=(w_1,w_2,\ldots,w_n)^t$ $\in \mbb{F}_p^n$, define $\langle \ul{u},\ul{w} \rangle=\us{i=1}{\os{n}{\sum}}u_iw_i\in \mbb{F}_p$.
Then the Heisenberg group is defined as a set $ES_1(p,n)=\mbb{F}_p^n \oplus \mbb{F}_p^n \oplus \mbb{F}_p$
with the following group operation. For $(\ul{u}^i,\ul{w}^i,z^i)\in ES_1(p,n),i=1,2$, 
\equ{(\ul{u}^1,\ul{w}^1,z^1).(\ul{u}^2,\ul{w}^2,z^2)=(\ul{u}^1+\ul{u}^2,\ul{w}^1+\ul{w}^2,z^1+z^2+\langle \ul{u}^1,\ul{w}^2\rangle).}
\end{definition}

\begin{definition}[Extra-special $p$-group of Second Type: Exponent $p^2$]
\label{defn:ESPII}
Let $p$ be an odd prime, $n$ be a positive integer and $\ZZ i$ be the cyclic ring of order $p^i,i=1,2$. Let 
$i_{21}:\ZZ 1=\{0,1,2,\ldots, p-1\} \hookrightarrow \ZZ 2=\{0,1,2,\ldots,p^2-1\}$ with $i_{21}(a)=pa$ for $a\in \ZZ 1$ be the standard inclusion as an abelian group where the generator $1\in \ZZ 1$ maps to $p\in \ZZ 2$. 
For $\ul{u}=(u_2,u_3,\ldots,$ $u_n)^t,\ul{w}=(w_2,w_3,\ldots,w_n)^t \in (\ZZ 1)^{n-1}$, define $\langle \ul{u},\ul{w}\rangle=\us{i=2}{\os{n}{\sum}}u_iw_i\in \ZZ 1$.
The extra-special group of second type is defined as a set\equ{ES_2(p,n)=\ZZ 2\oplus (\ZZ 1)^{n-1}\oplus (\ZZ 1) \oplus (\ZZ 1)^{n-1}} with the following group operation.  
For $(u^i_1,\ul{u}^i,w^i_1,\ul{w}^i)\in ES_2(p,n),i=1,2$, 
\equa{\big(u^1_1,\ul{u}^1,w^1_1,\ul{w}^1\big).&\big(u^2_1,\ul{u}^2,w^2_1,\ul{w}^2\big)=\\&\big(u^1_1+u^2_1+i_{21}(w^2_1)u^1_1+ i_{21}(\langle \ul{u}^1,\ul{w}^2\rangle),\ul{u}^1+\ul{u}^2,w^1_1+w^2_1,\ul{w}^1+\ul{w}^2\big).}
\end{definition}
\begin{definition}[Extra-special $p$-group and its associated symplectic form]
Let $p$ be an odd prime. A finite group $G$ is said to be an extra-special $p$-group if $[G,G]=G'=\mcl{Z}(G)$ and $\mcl{Z}(G)$ is of order $p$. In this case we have that $\frac{G}{\mcl{Z}(G)}$ is elementary abelian, isomorphic to $(\ZZ 1)^{2n}$ for some $n\in \mbb{N}$ and is equipped with non-degenerate symplectic form $\langle\langle *,* \rangle\rangle$ defined as:
\equ{\langle\langle *,* \rangle\rangle:\frac{G}{\mcl{Z}(G)} \times \frac{G}{\mcl{Z}(G)} \lra \mbb{F}_p, \langle\langle \ol{x},\ol{y}\rangle\rangle = f(x,y) \text{ with }\ol{x}=x\mcl{Z}(G),\ol{y}=y\mcl{Z}(G)}
where $f:G\times G \lra \mbb{F}_p$ is defined by the equation $[x,y]=z^{f(x,y)}$ for a fixed generator $z$ of $\mcl{Z}(G)$. Consequentially the group $G$ hence has order $p^{2n+1}$. If $\gs$ is an endomorphism (automorphism) of $G$ then it gives rise to $\ol{\gs}$ an endomorphism (automorphism) of $\frac{G}{\mcl{Z}(G)}$.  
\end{definition}
\begin{remark}
Let $p$ be an odd prime and $G$ be an extra-special $p$-group. Then $G$ is isomorphic to either $ES_1(p,n)$ or $ES_2(p,n)$ for some $n$. 
\end{remark}

\begin{definition}[Partial order on orbits and the notion of degeneration]
\label{defn:PODegeneration}
Let $G$ be a finite group. Let $Aut(G),End(G)$ be its automorphism group and endomorphism semigroup respectively. Let $S$ be the set of automorphism orbits in $G$. Let $x,y\in G$. We say $y$ is {\it endomorphic} to $x$ or $x$ degenerates to $y$ if there exists $\gs\in End(G)$ such that $\gs(x)=y$. We say $y$ is {\it automorphic} to $x$ if there exists $\gs\in Aut(G)$ such that $\gs(x)=y$. We say the endomorphism semigroup induces a {\it partial order $\leq $} on the automorphism orbits if $y$ is endomorphic to $x$ and $x$ is endomorphic to $y$ then $y$ is automorphic to $x$. In this case, if $O_1,O_2\in S$ are two orbits then we write $O_2 \leq O_1$ if for some 
$y\in O_2,x\in O_1$ we have $y$ is endomorphic to $x$.
\end{definition}
\begin{remark}
Let $p$ be a prime and $G$ be a finite abelian $p$-group. Then the endomorphism semigroup $End(G)$ (here an endomorphism algebra) induces a partial order on automorphism orbits~\cite{MR2793603}.
\end{remark}
Now we introduce some notation before stating the first main theorem.
\begin{itemize}
\item Let $e_i^n=(0,\ldots,0,1,0,\ldots,0)^t\in \mbb{F}_p^n$ be the vector with $1$ in the $i^{th}$ position and $0$ elsewhere. Here $t$ stands for transpose.
\item Let $\ul{0}^n=(0,\ldots,0)^t \in \mbb{F}_p^n$ be the zero vector.
\item $\ul{u},\ul{w}$ denote vectors in $\mbb{F}_p^n$ for some $n$.
\item Let $symp^{scalar}(2n,\mbb{F}_p)=\bigg\{N\in M_{2n}(\mbb{F}_p)\mid N^t\Gd N=l\Gd,l\in \mbb{F}_p,\\ \Gd=\mattwo{0_{n\times n}}{I_{n\times n}}{-I_{n\times n}}{0_{n\times n}}\bigg\}$.
\item Let $Sp^{scalar}(2n,\mbb{F}_p)=\bigg\{M\in GL_{2n}(\mbb{F}_p)\mid M^t\Gd M=l\Gd,l\in \mbb{F}_p^{*},\\ \Gd=\mattwo{0_{n\times n}}{I_{n\times n}}{-I_{n\times n}}{0_{n\times n}}\bigg\}$.
\end{itemize}

We state the first main theorem of the article.
\begin{thmOmega}
\namedlabel{theorem:ExtraSpecialTypeI}{$\Gom$}
~\\Let $p$ be an odd prime and $n$ be a positive integer. Let $G=ES_1(p,n)$. Then:
\begin{enumerate}[label=(\Alph*)]
\item If $\gs \in End(G)$ then the induced automorphism $\ol{\gs}$ of $\frac{G}{\mcl{Z}(G)}$ satisfies \equ{\langle\langle \ol{\gs}(\ol{x}),\ol{\gs}(\ol{y})\rangle\rangle = 
	l\langle\langle \ol{x},\ol{y}\rangle\rangle} where $l\in \mbb{F}_p$ given by the equation $\gs(z)=z^{l}$ for any generator $z$ of $\mcl{Z}(G)$.

\item The explicit expression for $\gs\in End(G)$ is given as follows. Consider the elements $x_i=(e_i^n,\ul{0}^n,0),y_i=(\ul{0}^n,e_i^n,0) \in G,1\leq i\leq n$.  
Let \equan{TypeI1}{A=[a_{ij}]_{n\times n},B&=[b_{ij}]_{n\times n},C=[c_{ij}]_{n\times n},D=[d_{ij}]_{n\times n} \text{ and }\\
	\ol{\gs}&=\mattwo ACDB\in symp^{scalar}(2n,\mbb{F}_p),\ol{\gs}^t\Gd\ol{\gs}=l\Gd,l\in \mbb{F}_p} 
with respect to the ordered basis
$\{\ol{x}_1,\ol{x}_2,\ldots,\ol{x}_n,\ol{y}_1,\ol{y}_2,\ldots,\ol{y}_n\}$ of $\frac{G}{\mcl{Z}(G)} = \mbb{F}_p^{2n}$. Then for $\ul{u}=(u_1,u_2,\ldots,u_n)^t,\ul{w}=(w_1,w_2,\ldots,w_n)^t\in \mbb{F}_p^n,z\in \mbb{F}_p$ we have 
\equan{TypeI2}{\gs(\ul{u},\ul{w},z)=(A\ul{u}+C\ul{w},D\ul{u}+B\ul{w},\widetilde{\gs}(\ul{u},\ul{w},z))}
where 
\equan{TypeI3}{\widetilde{\gs}(\ul{u},\ul{w},z)=\ga(\ul{u})+\gb(\ul{w})+lz+ \frac12\ul{u}^t(A^tD)\ul{u}
	+\frac12\ul{w}^t(C^tB)\ul{w}+\ul{w}^t(C^tD)\ul{u}}
for some $\ga,\gb\in (\mbb{F}_p^n)^{\vee}$ (dual of $\mbb{F}_p^n$) and $l\in \mbb{F}_p$ which satisfies the equation $\ol{\gs}^t\Gd\ol{\gs}=l\Gd$.  Conversely if $\gs$ is given as in Equations~\ref{Eq:TypeI1},~\ref{Eq:TypeI2},~\ref{Eq:TypeI3} then $\gs\in End(G)$.
	
\item If $\gs \in Aut(G)$ then the induced automorphism $\ol{\gs}$ of $\frac{G}{\mcl{Z}(G)}$ satisfies \equ{\langle\langle \ol{\gs}(\ol{x}),\ol{\gs}(\ol{y})\rangle\rangle = 
l\langle\langle \ol{x},\ol{y}\rangle\rangle} where $l\in \mbb{F}_p^{*}$ given by the equation $\gs(z)=z^{l}$ for any generator $z$ of $\mcl{Z}(G)$.

\item With the notations in $(B)$, the expression for an automorphism $\gs\in Aut(G)$ remains the same as in $(B)$ except that, here $\ol{\gs}\in Sp^{scalar}(2n,\mbb{F}_p)$ is invertible with $l\in \mbb{F}_p^{*}$. Conversely if $\gs$ is given as in Equations~\ref{Eq:TypeI1},~\ref{Eq:TypeI2},~\ref{Eq:TypeI3} and $l\neq 0$ then $\gs\in Aut(G)$.

\item The set of endomorphism semigroup images of an element $g\in G$ is given by:
\begin{enumerate}
\item $\{e\}$ if $g=e$ and has cardinality $1$.
\item $\mcl{Z}(G)$ if $g\in \mcl{Z}(G)\bs\{e\}$ and has cardinality $p$.
\item $G$ if $g\in G\bs \mcl{Z}(G)$ and has cardinality $p^{2n+1}$.
\end{enumerate}

\item There are three automorphism orbits in $G$. They are given by:
\begin{enumerate}
\item The identity element $\{e\}$ and has cardinality $1$.
\item The central non-identity elements $\mcl{Z}(G)\bs\{e\}$ and has cardinality $p-1$.
\item The non-central elements $G\bs \mcl{Z}(G)$ and has cardinality $p^{2n+1}-p$.
\end{enumerate}

\item The endomorphism semigroup induces a partial order (in fact a total order) on automorphism orbits which is given by 
\equ{\{e\} < \mcl{Z}(G)\bs \{e\} < G\bs \mcl{Z}(G).}
\end{enumerate}
\end{thmOmega}
Now we introduce some further notation before stating the second main theorem.
\begin{itemize}
\item $\widetilde{\ul{u}},\widetilde{\ul{w}}$ denote vectors in $\mbb{F}_p^n$ for some $n$.
\item Let $i_{21}:\ZZ 1 \hookrightarrow \ZZ 2$ be the inclusion of the abelian group $\ZZ 1$ taking the generator $1\in \ZZ 1$ to $p\in \ZZ 2$.
\item For $u_1\in \ZZ 2$, let $\ol{u}_1\in \ZZ 1$ be its reduction modulo $p$.
\item Let $\gp:\Z/p\Z\oplus (\ZZ 1)^{n-1}\lra (\ZZ 1)^{n-1}$ be the projection ignoring the first co-ordinate.
\item For $G=ES_2(p,n)$ let $H= p\big(\ZZ 2\big)\oplus (\ZZ 1)^{n-1}\oplus \ZZ 1\oplus (\ZZ 1)^{n-1}$,
$K=p\big(\ZZ 2\big)\oplus\{\ul{0}^{n-1}\}\oplus \ZZ 1\oplus\{\ul{0}^{n-1}\}=\mcl{Z}(H)$ and we have
$\mcl{Z}(G)=p\big(\ZZ 2\big)\oplus\{\ul{0}^{n-1}\}\oplus \{0\}\oplus\{\ul{0}^{n-1}\}$.
\end{itemize}
Now we state the second main theorem of the article.
\begin{thmSigma}
\namedlabel{theorem:ExtraSpecialTypeII}{$\Gs$}
~\\Let $p$ be an odd prime and $n$ be a positive integer. Let $G=ES_2(p,n)$. Then: 
\begin{enumerate}[label=(\Alph*)]
\item If $\gs \in End(G)$ then the induced endomorphism $\ol{\gs}$ of $\frac{G}{\mcl{Z}(G)}$ satisfies \equ{\langle\langle \ol{\gs}(\ol{x}),\ol{\gs}(\ol{y})\rangle\rangle = 
	l\langle\langle \ol{x},\ol{y}\rangle\rangle} where $l\in \mbb{F}_p$ given by the equation $\gs(z)=z^{l}$ for any generator $z$ of $\mcl{Z}(G)$. We also have
\begin{enumerate}[label=(\alph*)]
	\item $\gs(x_1)$ can be any element of $G$ where $x_1=(1,\ul{0}^{n-1},0,\ul{0}^{n-1})\in G$.
	\item For $2\leq i\leq n,1\leq j\leq n, \gs(x_i),\gs(y_j)\in H$ where $x_i=(0,e^{n-1}_{i-1},0,\ul{0}^{n-1})$, $y_i= (0,\ul{0}^{n-1},$ $0,e^{n-1}_{i-1})$.
\end{enumerate}
\item The explicit expression for $\gs\in End(G)$ is given as follows.  
Let \equan{TypeII4}{
	A=[a_{ij}]_{n\times n},B&=[b_{ij}]_{n\times n},C=[c_{ij}]_{n\times n},D=[d_{ij}]_{n\times n} \text{ and }\\
	\ol{\gs}&=\mattwo ACDB\in symp^{scalar}(2n,\mbb{F}_p)} 
with respect to the ordered basis
$\{\ol{x}_1,\ol{x}_2,\ldots,\ol{x}_n,\ol{y}_1,\ol{y}_2,\ldots,\ol{y}_n\}$ of $\frac{G}{\mcl{Z}(G)} = \mbb{F}_p^{2n}$.
For $(u_1,\ul{u},w_1,\ul{w})\in G$, let $\widetilde{\ul{u}}=\matcoltwo{\ol{u}_1}{\ul{u}}=(\widetilde{u}_1,\widetilde{u}_2,\ldots,\widetilde{u}_n)^t\in (\ZZ 1)^n,
\widetilde{\ul{w}}=\matcoltwo{w_1}{\ul{w}}=(\widetilde{w}_1,\widetilde{w}_2,\ldots,\widetilde{w}_n)^t\in (\ZZ 1)^n$. 
Then we have $\ol{\gs}$ may be non-invertible and
\equan{TypeII5}{\ol{\gs}^t\Gd\ol{\gs}&=a_{11}\Gd  (\text{ where } a_{11}\text{ can be zero}),\\
	a_{12}=a_{13}=\ldots&=a_{1n}=0,c_{11}=c_{12}=\ldots=c_{1n}=0 \text{ and }\\
	\gs(u_1,\ul{u},w_1,\ul{w})&=(au_1+i_{21}(s),\gp(A\widetilde{\ul{u}}+C\widetilde{\ul{w}}),D\widetilde{\ul{u}}+B\widetilde{\ul{w}})}where 
\equan{TypeII6}{a&\in (\ZZ 2)\text{ and }a \equiv a_{11}\mod p \text{ can be zero, }\\
	s&=\ga(\ul{u})+\gb(\widetilde{\ul{w}})+ \frac12\widetilde{\ul{u}}^t(A^tD)\widetilde{\ul{u}}+
	\frac12\widetilde{\ul{w}}^t(C^tB)\widetilde{\ul{w}}+\widetilde{\ul{w}}^t(C^tD)\widetilde{\ul{u}}}
for some $\ga \in ((\ZZ 1)^{n-1})^{\vee},\gb\in ((\ZZ 1)^n)^{\vee}$.
Conversely if $\gs$ is given as in Equations~\ref{Eq:TypeII4},~\ref{Eq:TypeII5},~\ref{Eq:TypeII6} then $\gs\in End(G)$.
\item If $\gs \in Aut(G)$ then the induced automorphism $\ol{\gs}$ of $\frac{G}{\mcl{Z}(G)}$ satisfies \equ{\langle\langle \ol{\gs}(\ol{x}),\ol{\gs}(\ol{y})\rangle\rangle = 
l\langle\langle \ol{x},\ol{y}\rangle\rangle} where $l\in \mbb{F}_p^{*}$ given by the equation $\gs(z)=z^{l}$ for any generator $z$ of $\mcl{Z}(G)$. We also have
 
\begin{enumerate}[label=(\alph*)]
\item $\gs(x_1)=x_1^lg$ for some $g\in H$.
\item $\gs(y_1)=y_1h$ for some $h\in \mcl{Z}(G)$.
\item For $2\leq i\leq n, \gs(x_i),\gs(y_i)\in H\bs K$.
\end{enumerate}
\item With the same notations in $(B)$ the expression for $\gs\in Aut(G)$ is given as follows.
Here 
\equan{TypeII1}{\ol{\gs}=\mattwo ACDB\in Sp^{scalar}(2n,\mbb{F}_p)}
and we have 
\equan{TypeII2}{
	\ol{\gs}^t\Gd\ol{\gs}&=a_{11}\Gd,a_{11}\in \mbb{F}_p^{*}, \text{ that is, }a_{11} \not\equiv 0\mod p,\\
	a_{12}=a_{13}=\ldots&=a_{1n}=0,c_{11}=c_{12}=\ldots=c_{1n}=0 \text{ and }\\
	\gs(u_1,\ul{u},w_1,\ul{w})&=(au_1+i_{21}(s),\gp(A\widetilde{\ul{u}}+C\widetilde{\ul{w}}),D\widetilde{\ul{u}}+B\widetilde{\ul{w}})}
where 
\equan{TypeII3}{a&\in (\ZZ 2)^{*}\text{ and }a \equiv a_{11}\mod p,\\
	s&=\ga(\ul{u})+\gb(\widetilde{\ul{w}})+ \frac12\widetilde{\ul{u}}^t(A^tD)\widetilde{\ul{u}}+
\frac12\widetilde{\ul{w}}^t(C^tB)\widetilde{\ul{w}}+\widetilde{\ul{w}}^t(C^tD)\widetilde{\ul{u}}}
for some $\ga \in ((\ZZ 1)^{n-1})^{\vee},\gb\in ((\ZZ 1)^n)^{\vee}$.
Conversely if $\gs$ is given as in Equations~\ref{Eq:TypeII1},~\ref{Eq:TypeII2},~\ref{Eq:TypeII3} then $\gs\in Aut(G)$.

As a consequence we have in addition
\begin{enumerate}[label=(\alph*)]
\item $b_{11}=1$.
\item $b_{21}=b_{31}=\ldots =b_{n1}=c_{21}=c_{31}=\ldots =c_{n1}=0$.
\end{enumerate}
\item The set of endomorphism semigroup images of an element $g\in G$ is given by:
\begin{enumerate}
	\item $\{e\}$ if $g=e$ and has cardinality $1$.
	\item $\mcl{Z}(G)$ if $g\in \mcl{Z}(G)\bs\{e\}$ and has cardinality $p$.
	\item $H$ if $g\in H\bs \mcl{Z}(G)$ and has cardinality $p^{2n}$.
	\item $G$ if $g\in G\bs H$ and has cardinality $p^{2n+1}$.
\end{enumerate}
\item There are $(p+2)$ automorphism orbits if $n=1$ and $(p+3)$ automorphism orbits if $n>1$. They are given by:
\begin{enumerate}[label=(\alph*)]
\item The identity element $\{e\}$ and has cardinality $1$.
\item The central non-identity elements $\mcl{Z}(G)\bs \{e\}$ and has cardinality $p-1$.
\item For $b\in (\ZZ 1)^{*}, \mcl{O}_b=p(\ZZ 2)\times \{\ul{0}^{n-1}\}\times\{b\}\times \{\ul{0}^{n-1}\}$ and has cardinality $p$.
\item $G\bs H$, that is, all elements of order $p^2$ and has cardinality $p^{2n+1}-p^{2n}$.
\item if $n>1$ then we have one more orbit $H\bs K$ and has cardinality $p^{2n}-p^2$.
\end{enumerate}
\item In this group, there exist two elements which are endomorphic to each other but they are not automorphic. The endomorphism semigroup does not induce a partial order on automorphism orbits.
In particular the set \equ{H\bs \mcl{Z}(G) = \us{b\in (\ZZ 1)^{*}}\bigsqcup \mcl{O}_b \bigsqcup (H\bs K)}
is a disjoint union of $p$ automorphism orbits.
\end{enumerate}
\end{thmSigma}

\section{\bf{Preliminaries}}
\label{sec:Preliminaries}
It is well known that any extra-special $p$-group has exponent either $p$ or $p^2$
and has order $p^{2n+1}$ for some $n\in \mbb{N}$ (refer to D.~J.~S.~Robinson~\cite{MR648604}, Chapter $5$, pp. 140-142). For an odd prime $p$, if an extra-special $p$-group of order $p^{2n+1}$ is of exponent $p$ then it is isomorphic to $ES_1(p,n)$ and if it is of exponent $p^2$ then it is isomorphic to $ES_2(p,n)$. We also give one more way of presenting the group $ES_i(p,n)$ using a symplectic form for $i=1,2$ which will be useful to prove certain results.

\begin{definition}[Alternative Definition for $ES_1(p,n)$]
Let $p$ be an odd prime. Let $\widetilde{ES}_1(p,n)=\mbb{F}_p^n\oplus \mbb{F}_p^n\oplus \mbb{F}_p$. Let $\langle\langle *,*\rangle \rangle$ be a non-degenerate symplectic bilinear form on $\mbb{F}_p^{2n}$. Then the group structure on $\widetilde{ES}_1(p,n)$ is defined as: For $(\ul{u}^i,\ul{w}^i,z^i)\in \widetilde{ES}_1(p,n),i=1,2$ we have \equ{(\ul{u}^1,\ul{w}^1,z^1).(\ul{u}^2,\ul{w}^2,z^2)=\bigg(\ul{u}^1+\ul{u}^2,\ul{w}^1+\ul{w}^2,z^1+z^2+\frac 12\bigg\langle\bigg\langle\matcoltwo{\ul{u}^1}{\ul{w}^1},\matcoltwo{\ul{u}^2}{\ul{w}^2}\bigg\rangle\bigg\rangle\bigg).}
\end{definition}
\begin{definition}[Alternative Definition for $ES_2(p,n)$]
Let $p$ be an odd prime, $n$ be a positive integer and $\ZZ i$ be the cyclic ring of order $p^i,i=1,2$. Let 
$i_{21}:\ZZ 1=\{0,1,2,\ldots p-1\} \hookrightarrow \ZZ 2=\{0,1,2,\ldots,p^2-1\}$ with $i_{21}(a)=pa$ for $a\in \ZZ 1$ be the standard inclusion as an abelian group where the generator $1\in \ZZ 1$ maps to $p\in \ZZ 2$. Let\equ{\widetilde{ES}_2(p,n)=\ZZ 2\oplus (\ZZ 1)^{n-1}\oplus (\ZZ 1) \oplus (\ZZ 1)^{n-1}.} Then the group structure on $\widetilde{ES}_2(p,n)$ is defined as follows.
Let $\langle\langle *,*\rangle \rangle$ be the non-degenerate symplectic bilinear form on $(\ZZ 1)^{2n}$ given by the matrix $J=\mattwo{0_{n\times n}}{I_{n\times n}}{-I_{n\times n}}{0_{n\times n}}$
with respect to the standard basis.
Let $(u^i_1,\ul{u}^i,w^i_1,\ul{w}^i)\in ES_2(p,n),i=1,2$. Let $\widetilde{\ul{u}}^i=\matcoltwo{\ol{u}^i_1}{\ul{u}^i},\widetilde{\ul{w}}^i=\matcoltwo{w^i_1}{\ul{w}^i}\in (\ZZ 1)^n$ for $i=1,2$
where $\ol{u}^i_1$ is reduction of $u^i_1$ modulo $p$. Then 
\equa{\big(u^1_1,\ul{u}^1,w^1_1,\ul{w}^1\big).&\big(u^2_1,\ul{u}^2,w^2_1,\ul{w}^2\big)=\\&\bigg(u^1_1+u^2_1+i_{21}\bigg(\bigg\langle\bigg\langle\matcoltwo{\widetilde{\ul{u}}^1}{\widetilde{\ul{w}}^1},\matcoltwo{\widetilde{\ul{u}}^2}{\widetilde{\ul{w}}^2} \bigg\rangle\bigg\rangle\bigg),\ul{u}^1+\ul{u}^2,w^1_1+w^2_1,\ul{w}^1+\ul{w}^2\bigg).}
\end{definition}	
Here we state the theorem.
\begin{thm}
\label{theorem:ExtraSpecialSymplectic}
$ES_l(p,n)\cong \widetilde{ES}_l(p,n),l=1,2$.
\end{thm}
\begin{proof}
We prove for $l=1$ first. Let $\ul{u}^i=(u^i_1,u^i_2,\ldots,u^i_n)^t,\ul{w}^i=(w^i_1,w^i_2,\ldots,w^i_n)^t \in \mbb{F}_p^n,i=1,2$.
Let $\ul{u}=(u_1,u_2,\ldots,u_n)^t,\ul{w}=(w_1,w_2,\ldots,w_n)^t\in \mbb{F}_p^n$. Let $\langle\ul{u},\ul{w}\rangle=
\us{j=1}{\os{n}{\sum}}u_jw_j\in \mbb{F}_p$. Let us fix the symplectic form as  \equ{\bigg\langle\bigg\langle \matcoltwo{\ul{u}^1}{\ul{w}^1}, \matcoltwo{\ul{u}^2}{\ul{w}^2}\bigg\rangle\bigg\rangle=\us{j=1}{\os{n}{\sum}}\big(u^1_jw^2_j-u^2_jw^1_j\big)=\langle\ul{u}^1,\ul{w}^2\rangle-\langle\ul{u}^2,\ul{w}^1\rangle.}
Define a map $\gl:\widetilde{ES}_1(p,n)\lra ES_1(p,n)$ given by 
\equ{\gl(\ul{u},\ul{w},z)=(\ul{u},\ul{w},z+\frac 12\big(\us{j=1}{\os{n}{\sum}}u_jw_j\big))=\big(\ul{u},\ul{w},z+\frac 12\langle\ul{u},\ul{w}\rangle\big).}
It is easy to check that $\gl$ is an isomorphism.

Now we prove for $l=2$. For $i=1,2$ let $u_1^i\in \ZZ 2,w^i_1\in \ZZ 1,\ul{u}^i,
\ul{w}^i\in (\ZZ 1)^{n-1}$. For $i=1,2$ let $\widetilde{\ul{u}}^i=\matcoltwo{\ol{u}_1^i}{\ul{u}^i}=(\widetilde{u}^i_1,\widetilde{u}^i_2,\ldots,\widetilde{u}^i_n)^t,
\widetilde{\ul{w}}^i=\matcoltwo{w_1^i}{\ul{w}^i}=(\widetilde{w}^i_1,\widetilde{w}^i_2,\ldots,\widetilde{w}^i_n)^t\in (\ZZ 1)^n$ where $\ol{u}^i_1$ is reduction modulo $p$ of $u^i_1\in \ZZ 2$. Let $u_1\in \ZZ 2,w_1\in \ZZ 1,\ul{u},\ul{w}\in (\ZZ 1)^{n-1}$. Let $\widetilde{\ul{u}}=\matcoltwo{\ol{u}_1}{\ul{u}}=(\widetilde{u}_1,\widetilde{u}_2,\ldots,$ $\widetilde{u}_n)^t,
\widetilde{\ul{w}}=\matcoltwo{\ol{w}_1}{\ul{w}}=(\widetilde{w}_1,\widetilde{w}_2,\ldots,$ $\widetilde{w}_n)^t\in (\ZZ 1)^{n}$.
Let $\langle \widetilde{\ul{u}},\widetilde{\ul{w}}\rangle=\us{j=1}{\os{n}{\sum}}\widetilde{u}_j\widetilde{w}_j\in \ZZ 1$.
The symplectic form is given as 
\equ{\bigg\langle\bigg\langle \matcoltwo{\widetilde{\ul{u}}^1}{\widetilde{\ul{w}}^1}, \matcoltwo{\widetilde{\ul{u}}^2}{\widetilde{\ul{w}}^2}\bigg\rangle\bigg\rangle=\us{j=1}{\os{n}{\sum}}\big(\widetilde{u}^1_j\widetilde{w}^2_j-\widetilde{u}^2_j\widetilde{w}^1_j\big)=\langle\widetilde{\ul{u}}^1,\widetilde{\ul{w}}^2\rangle-\langle\widetilde{\ul{u}}^2,\widetilde{\ul{w}}^1\rangle.}
Define a map $\gd:\widetilde{ES}_2(p,n) \lra ES_2(p,n)$ given by
\equ{\gd(u_1,\ul{u},w_1,\ul{w})=\big(u_1+\frac 12i_{21}(\langle \widetilde{\ul{u}},\widetilde{\ul{w}} \rangle),\ul{u},w_1,\ul{w}\big).}
It is easy to check that $\gd$ is an isomorphism.
This completes the proof of the theorem.
\end{proof}
Now we prove a general proposition regarding extra-special $p$-groups.
\begin{prop}
\label{prop:Symplectic}
Let $G$ be an extra-special $p$-group. Let $z\in \mcl{Z}(G)$ be a generator such that $[g_1,g_2]=z^{f(g_1,g_2)}$ for $g_1,g_2\in G$ and $f:G\times G \lra \mbb{F}_p$. Let $\ol{f}:\frac{G}{\mcl{Z}(G)}\times \frac{G}{\mcl{Z}(G)}\lra \mbb{F}_p$ be its associated non-degenerate symplectic bilinear form defined as $\ol{f}(\ol{g_1},\ol{g_2})=f(g_1,g_2)$. Then we have:
\begin{enumerate}
\item For $\gs\in End(G)$, $\ol{f}(\ol{\gs}(\ol{g_1}),\ol{\gs}(\ol{g_2}))=l\ol{f}(\ol{g_1},\ol{g_2})$ for any $g_1,g_2\in G$ where $\gs(z)=z^l,l\in \mbb{F}_p$ and $\ol{\gs}$ is the induced endomorphism of $\frac{G}{\mcl{Z}(G)}$.
\item For $\gs\in Aut(G)$, $\ol{f}(\ol{\gs}(\ol{g_1}),\ol{\gs}(\ol{g_2}))=l\ol{f}(\ol{g_1},\ol{g_2})$ for any $g_1,g_2\in G$ where $\gs(z)=z^l,l\in \mbb{F}_p^{*}$ and $\ol{\gs}$ is the induced automorphism of $\frac{G}{\mcl{Z}(G)}$.
\end{enumerate}
\end{prop}
\begin{proof}
We have \equ{z^{lf(g_1,g_2)}=\gs(z^{f(g_1,g_2)})=\gs[g_1,g_2]=[\gs(g_1),\gs(g_2)]=z^{f(\gs(g_1),\gs(g_2))}.}
Now the proposition follows.
\end{proof}
\subsection{Some Commutative Diagrams on Extra-special $p$-Groups}
\label{sec:CommDiag}
Now we show that certain diagrams of groups and maps for the extra-special $p$-group of the first type are commutative.
First we observe that $\mcl{Z}(ES_1(p,n))=\{\ul{0}^n\}\oplus\{\ul{0}^n\}\oplus \mbb{F}_p=\mcl{Z}(\widetilde{ES}_1(p,n))$.
Let \equa{\gp_1: ES_1(p,n)=\mbb{F}_p^n\oplus \mbb{F}_p^n\oplus \mbb{F}_p &\lra \frac{ES_1(p,n)}{\mcl{Z}(ES_1(p,n))}=\mbb{F}_p^n\oplus \mbb{F}_p^n,\\ \widetilde{\gp}_1:\widetilde{ES}_1(p,n)=\mbb{F}_p^n\oplus \mbb{F}_p^n\oplus \mbb{F}_p &\lra \frac{\widetilde{ES}_1(p,n)}{\mcl{Z}(\widetilde{ES}_1(p,n))}=\mbb{F}_p^n\oplus \mbb{F}_p^n} be the quotient maps of groups.
Let the induced maps be \equa{\Gf_1:Aut(ES_1(p,n)) &\lra Aut\bigg(\frac{ES_1(p,n)}{\mcl{Z}(ES_1(p,n))}\bigg)=GL_{2n}(\mbb{F}_p),\\
	\widetilde{\Gf}_1:Aut(\widetilde{ES}_1(p,n)) &\lra Aut\bigg(\frac{\widetilde{ES}_1(p,n)}{\mcl{Z}(\widetilde{ES}_1(p,n))}\bigg)=GL_{2n}(\mbb{F}_p).}
Then the following two diagrams commute.
\[
\begin{tikzcd}
0 \arrow[r,""]& \mbb{F}_p=\mcl{Z}(\widetilde{ES}_1(p,n)) \arrow[hook]{r}{} \arrow[]{d}{\Vert}[swap]{Id} & \widetilde{ES}_1(p,n) \arrow[two heads]{r}{\widetilde{\gp}_1} \arrow[]{d}{}[swap]{\gl} & \mbb{F}_p^n\oplus \mbb{F}_p^n \arrow[r,""] \arrow[]{d}{\Vert}[swap]{Id} & 0\\
0 \arrow[r,""]& \mbb{F}_p=\mcl{Z}(ES_1(p,n))\arrow[hook]{r}{} & ES_1(p,n)\arrow[two heads]{r}{\gp_1}& \mbb{F}_p^n\oplus \mbb{F}_p^n \arrow[r,""] & 0
\end{tikzcd}
\]
\equan{ComDiagII}{}
\[
\begin{tikzcd}
 Aut(\widetilde{ES}_1(p,n)) \arrow[]{r}{\widetilde{\Gf}_1} \arrow[]{d}{\cong}[swap]{\gl\circ(*)\circ\gl^{-1}} & GL_{2n}(\mbb{F}_p) \arrow[]{d}{\Vert}[swap]{Id}\\
 Aut(ES_1(p,n)) \arrow[]{r}{\Gf_1} & GL_{2n}(\mbb{F}_p)
\end{tikzcd}
\]
Here $\gl$ is as defined in the proof of Theorem~\ref{theorem:ExtraSpecialSymplectic}.
In particular we get that $Im(\widetilde{\Gf}_1)=Im(\Gf_1)\subs GL_{2n}(\mbb{F}_p)$.
\begin{prop}
$Im(\widetilde{\Gf}_1)=Im(\Gf_1)=Sp^{scalar}(2n,\mbb{F}_p)$.
\end{prop}
\begin{proof}
For $\ol{\gs}\in Sp^{scalar}(2n,\mbb{F}_p)$ we can define an automorphism $\gs\in Aut(\widetilde{ES}_1(p,n))$ such that $\widetilde{\Gf}_1(\gs)=\ol{\gs}$ as follows.
\equ{\gs(\ul{v},z)=(\ol{\gs}(\ul{v}),lz) \text{ where }\ol{\gs}^t\Gd\ol{\gs}=l\Gd,(\ul{v},z)\in \mbb{F}_p^{2n}\oplus \mbb{F}_p=\widetilde{ES}_1(p,n).} Hence we have 
$Sp^{scalar}(2n,\mbb{F}_p)\subseteq Im(\widetilde{\Gf}_1)=Im(\Gf_1)\subs GL_{2n}(\mbb{F}_p)$. 
Now use Proposition~\ref{prop:Symplectic} to conclude equality.
\end{proof}

Now we show that certain diagrams of groups and maps for the extra-special $p$-group of the second type are commutative.
First we observe that $\mcl{Z}(ES_2(p,n))= p(\ZZ 2) \oplus \{\ul{0}^{n-1}\} \oplus \{0\} \oplus \{\ul{0}^{n-1}\}=\mcl{Z}(\widetilde{ES}_2(p,n))$. 
Let \equa{\gp_2: ES_2(p,n)&=(\ZZ 2) \oplus (\ZZ 1)^{n-1} \oplus (\ZZ 1) \oplus (\ZZ 1)^{p-1}\lra\\ 
	\frac{ES_2(p,n)}{\mcl{Z}(ES_2(p,n))}&=(\ZZ 1) \oplus (\ZZ 1)^{n-1} \oplus (\ZZ 1) \oplus (\ZZ 1)^{p-1}=(\ZZ 1)^{2n},\\ \widetilde{\gp}_2: \widetilde{ES}_2(p,n)&=(\ZZ 2) \oplus (\ZZ 1)^{n-1} \oplus (\ZZ 1) \oplus (\ZZ 1)^{p-1}\lra\\ 
	\frac{\widetilde{ES}_2(p,n)}{\mcl{Z}(\widetilde{ES}_2(p,n))}&=(\ZZ 1) \oplus (\ZZ 1)^{n-1} \oplus (\ZZ 1) \oplus (\ZZ 1)^{p-1}=(\ZZ 1)^{2n},} be the quotient maps of groups. Let the induced maps be
\equa{\Gf_2:Aut(ES_2(p,n)) &\lra Aut\bigg(\frac{ES_2(p,n)}{\mcl{Z}(ES_2(p,n))}\bigg)=GL_{2n}(\ZZ 1),\\
	\widetilde{\Gf}_2:Aut(\widetilde{ES}_2(p,n)) &\lra Aut\bigg(\frac{\widetilde{ES}_2(p,n)}{\mcl{Z}(\widetilde{ES}_2(p,n))}\bigg)=GL_{2n}(\ZZ 1).}
Then the following two diagrams commute.

\[
\begin{tikzcd}
0 \arrow[r,""]& p(\ZZ 2)=\mcl{Z}(\widetilde{ES}_2(p,n)) \arrow[hook]{r}{} \arrow[]{d}{\Vert}[swap]{Id} & \widetilde{ES}_2(p,n) \arrow[two heads]{r}{\widetilde{\gp}_2} \arrow[]{d}{}[swap]{\gd} & (\ZZ 1)^{2n} \arrow[r,""] \arrow[]{d}{\Vert}[swap]{Id} & 0\\
0 \arrow[r,""]& p(\ZZ 2)=\mcl{Z}(ES_2(p,n))\arrow[hook]{r}{} & ES_2(p,n)\arrow[two heads]{r}{\gp_2}& (\ZZ 1)^{2n} \arrow[r,""] & 0
\end{tikzcd}
\]
\equan{ComDiagII}{}
\[
\begin{tikzcd}
Aut(\widetilde{ES}_2(p,n)) \arrow[]{r}{\widetilde{\Gf}_2} \arrow[]{d}{\cong}[swap]{\gd\circ(*)\circ\gd^{-1}} & GL_{2n}(\ZZ 1) \arrow[]{d}{\Vert}[swap]{Id}\\
Aut(ES_2(p,n)) \arrow[]{r}{\Gf_2} & GL_{2n}(\ZZ 1)
\end{tikzcd}
\]
Here $\gd$ is as defined in the proof of Theorem~\ref{theorem:ExtraSpecialSymplectic}.
In particular we get that $Im(\widetilde{\Gf}_2)=Im(\Gf_2)\subs GL_{2n}(\ZZ 1)$. We describe this image exactly in Proposition~\ref{prop:ImageES2}.
\section{\bf{Proof of the First Main Theorem}}
In this section we prove first main Theorem~\ref{theorem:ExtraSpecialTypeI}.
\begin{proof}
Here $G=ES_1(p,n)$. Let $\gs\in End(G)$ and $\ol{\gs}\in End(\frac{G}{\mcl{Z}(G)})= M_{2n}(\mbb{F}_p)$.  Let $\ol{\gs}=\mattwo ACDB$ with $A,B,C,D\in M_{n}(\mbb{F}_p)$.
Hence we have $\gs(\ul{u},\ul{w},z)=(A\ul{u}+C\ul{w},D\ul{u}+B\ul{w},\widetilde{\gs}(\ul{u},\ul{w},z))$ for some $\widetilde{\gs}:G\lra \mbb{F}_p$ for $(\ul{u},\ul{w},z)\in G$.
Using Proposition~\ref{prop:Symplectic} we have \equ{\ol{\gs}=\mattwo ACDB\in symp^{scalar}(2n,\mbb{F}_p)} and $A^tB-D^tC=l.\text{Id}_{n\times n}$ where $\ol{\gs}^t\Gd\ol{\gs}=l\Gd$. So we also have $ A^tD=D^tA,C^tB=B^tC$.
This computation does not give the explicit form of $\gs$ as we do not know $\widetilde{\gs}$.

Now we compute the explicit form of $\widetilde{\gs}$.
The homomorphism condition gives us that, for $(\ul{u}^i,\ul{w}^i,z^i)\in G,i=1,2$, 

\equan{Hom1}{\widetilde{\gs}(\ul{u}^1+\ul{u}^2,\ul{w}^1+\ul{w}^2,z^1+z^2+\langle \ul{u}^1,\ul{w}^2\rangle)=
\widetilde{\gs}(\ul{u}^1,\ul{w}^1,z^1)+&\widetilde{\gs}(\ul{u}^2,\ul{w}^2,z^2)+\\&\langle A\ul{u}^1+C\ul{w}^1,D\ul{u}^2+B\ul{w}^2\rangle.}

Putting $\ul{w}^1=\ul{w}^2=\ul{0}^n,z^1=z^2=0$ we get that \equan{Sym11}{\widetilde{\gs}(\ul{u}^1+\ul{u}^2,\ul{0}^n,0)=\widetilde{\gs}(\ul{u}^1,\ul{0}^n,0)+\widetilde{\gs}(\ul{u}^2,\ul{0}^n,0)+\langle A\ul{u}^1,D\ul{u}^2\rangle.} 
Similarly we have 
\equan{Sym13}{\widetilde{\gs}(\ul{0}^n,\ul{w}^1+\ul{w}^2,0)=\widetilde{\gs}(\ul{0}^n,\ul{w}^1,0)+\widetilde{\gs}(\ul{0}^n,\ul{w}^2,0)+\langle C\ul{w}^1,B\ul{w}^2\rangle.}
We conclude the following.
\begin{itemize}
\item $\widetilde{\gs}(\ul{0}^n,\ul{0}^n,0)=0$.
\item Since $(\ul{u},\ul{w},z)=(\ul{0}^n,\ul{w},z).(\ul{u},\ul{0}^n,0)$ and $(\ul{0}^n,\ul{w},z)=(\ul{0}^n,\ul{w},0).(\ul{0}^n,\ul{0}^n,z)$ we have 
from Equation~\ref{Eq:Hom1} that 
\equan{Sym15}{\widetilde{\gs}(\ul{u},\ul{w},z)&=\widetilde{\gs}(\ul{0}^n,\ul{w},z)+\widetilde{\gs}(\ul{u},\ul{0}^n,0)+\langle C\ul{w},D\ul{u}\rangle\\
&=\widetilde{\gs}(\ul{u},\ul{0}^n,0)+\widetilde{\gs}(\ul{0}^n,\ul{w},0)+\widetilde{\gs}(\ul{0}^n,\ul{0}^n,z)+\langle C\ul{w},D\ul{u}\rangle.}
\item If we define $\widetilde{\gs}_1(\ul{u})=\widetilde{\gs}(\ul{u},\ul{0}^n,0)-\frac12\langle A\ul{u},D\ul{u}\rangle$ then from Equation~\ref{Eq:Sym11} and $A^tD=D^tA$ we conclude that $\widetilde{\gs}_1(\ul{0}^n)=0,\widetilde{\gs}_1(\ul{u}^1+\ul{u}^2)=\widetilde{\gs}_1(\ul{u}^1)+\widetilde{\gs}_1(\ul{u}^2)$. Hence 
\equan{Sym16}{\widetilde{\gs}(\ul{u},\ul{0}^n,0)=\ga(\ul{u})+\frac 12\langle A\ul{u},D\ul{u}\rangle \text{ for some }\ga\in (\mbb{F}_p^n)^{\vee}.}
\item Similarly from Equation~\ref{Eq:Sym13} and $C^tB=B^tC$ we conclude that 
\equan{Sym17}{\widetilde{\gs}(\ul{0}^n,\ul{w},0)=\gb(\ul{w})+\frac 12\langle C\ul{w},B\ul{w}\rangle \text{ for some }\gb\in (\mbb{F}_p^n)^{\vee}.}
\item We observe that \equan{Sym18}{\widetilde{\gs}(\ul{0}^n,\ul{0}^n,z^1+z^2)&=\widetilde{\gs}(\ul{0}^n,\ul{0}^n,z^1)+\widetilde{\gs}(\ul{0}^n,\ul{0}^n,z^2)\\ \Ra \widetilde{\gs}(\ul{0}^n,\ul{0}^n,z)&=lz\text{ for some }l\in \mbb{F}_p.}
\item From Equations~\ref{Eq:Sym15},~\ref{Eq:Sym16},~\ref{Eq:Sym17},~\ref{Eq:Sym18} we conclude that 
\equan{Sym19}{\widetilde{\gs}(\ul{u},\ul{w},z)=\ga(\ul{u})+\gb(\ul{w})+lz+\frac 12\langle A\ul{u},D\ul{u}\rangle+\frac 12\langle C\ul{w},B\ul{w}\rangle+\langle C\ul{w},D\ul{u}\rangle} 
for some $\ga,\gb\in (\mbb{F}_p^n)^{\vee},l\in \mbb{F}_p$.
\end{itemize}
Conversely if $\ol{\gs}=\mattwo ACDB\in symp^{scalar}(2n,\mbb{F}_p)$ with $ \ol{\gs}^t\Gd\ol{\gs}=l\Gd$ and Equation~\ref{Eq:Sym19} holds, then it is clear that Equation~\ref{Eq:Hom1} holds and $\gs$ is an endomorphism of $G$. This proves (A),(B) in Theorem~\ref{theorem:ExtraSpecialTypeI}.

In case of $Aut(G)$, the proof is similar except that here for $\gs\in Aut(G)$, we have $l\in \mbb{F}_p^{*}$, that is, it is not allowed to be zero. This proves (C),(D) in Theorem~\ref{theorem:ExtraSpecialTypeI}.

Now we prove (E). In case $\gs\in End(G)$ we allow $l$ to be zero. Using Equations~\ref{Eq:TypeI2},~\ref{Eq:TypeI3}, we conclude that the endomorphism semigroup image of $g\in G$ is given by (a) $\{e\}$ if $g=e$, (b) $\mcl{Z}(G)$ if $g\in \mcl{Z}(G)\bs\{e\}$, (c) $G$ if $g\in G\bs \mcl{Z}(G)$.

Now we prove (F). Using Equations~\ref{Eq:TypeI2},~\ref{Eq:TypeI3} we conclude that there are three automorphism orbits as follows. The identity element $\{e\}$ is clearly an orbit. The non-identity central elements $\mcl{Z}(G)\bs \{e\}$ form an orbit, as automorphisms act transitively on the non-identity central elements because we can choose any non-zero value for $l$. Now the non-central elements $G\bs \mcl{Z}(G)$ form an orbit as the group $Sp^{scalar}(2n,\mbb{F}_p)$ acts transitively on $\mbb{F}_p^{2n}\bs\{\ul{0}^{2n}\}$ and using inner automorphisms we can change the central co-ordinate to any central co-ordinate for the non-central elements.

Now it is clear that endomorphism semigroup $End(G)$ induces a partial order (total order) on the automorphism orbits.
This proves (G) and thereby completes the proof of first main Theorem~\ref{theorem:ExtraSpecialTypeI}.
\end{proof}
Using first main Theorem~\ref{theorem:ExtraSpecialTypeI} we have the following corollary.
\begin{cor}
\label{cor:ImageES1}
Let $G=ES_1(p,n)$.
\begin{enumerate}
\item $\gs \in Aut(G)$ is an inner-automorphism if and only if $\ol{\gs}=$Id$_{2n\times 2n}$.
In this case $\widetilde{\gs}(\ul{u},\ul{w},z)=\ga(\ul{u})+\gb(\ul{w})+z$ for some $\ga,\gb\in (\mbb{F}_p^n)^{\vee}$ for any $(\ul{u},\ul{w},z) \in G$. 
\item We have an exact sequence 
\equ{1\lra \frac{G}{\mcl{Z}(G)}\cong Inn(G)\hookrightarrow Aut(G) \lra Sp^{scalar}(2n,\mbb{F}_p)\lra 1.}
\item \equa{\mid Aut(G) \mid&=p^{2n}\mid Sp^{scalar}(2n,\mbb{F}_p)\mid\\
	&=p^{2n}(p-1)\mid Sp(2n,\mbb{F}_p)\mid=p^{n^2+2n}(p-1)\us{j=1}{\os{n}{\prod}}(p^{2j}-1).}
\end{enumerate}
\end{cor}
The cardinality of $End(G)$ for $G=ES_1(p,n)$ is computed in Section~\ref{sec:Comb}, Theorem~\ref{theorem:CombExtraSpecial}.

\section{\bf{Proof of the Second Main Theorem}}
In this section we prove second main Theorem~\ref{theorem:ExtraSpecialTypeII}.
\begin{proof}
Here $G=ES_2(p,n)$. Let $\gs\in End(G)$ and $\ol{\gs}\in End(\frac{G}{\mcl{Z}(G)})= M_{2n}(\mbb{F}_p)$. Let \equ{\ol{\gs}=\mattwo ACDB\text{ with }A=[a_{ij}],B=[b_{ij}],C=[c_{ij}],D=[d_{ij}]\in M_{n}(\mbb{F}_p).}
Then for $x_1=(1,\ul{0}^{n-1},0,\ul{0}^{n-1}), \gs(x_1)=(a_{11},\ul{0}^{n-1},0,\ul{0}^{n-1}).g$ for some element $g\in H$. So for $z=(p,\ul{0}^{n-1},0,\ul{0}^{n-1})\in \mcl{Z}(G)$ we have $\gs(z)=(a_{11}p,\ul{0}^{n-1},0,\ul{0}^{n-1})$.
Now using Proposition~\ref{prop:Symplectic} we have $\ol{\gs}\in symp^{scalar}(2n,\mbb{F}_p)$ and $A^tB-D^tC=a_{11}.\text{Id}_{n\times n}$ where $\ol{\gs}^t\Gd\ol{\gs}=a_{11}\Gd$. We also have $A^tD=D^tA,C^tB=B^tC$.

Since the order of $x_1$ is $p^2$ we have $o(\gs(x_1))= p^2 \Llra a_{11}\not\equiv 0\mod p$. Since the order of $x_i=(0,e^{n-1}_{i-1},0,\ul{0}^{n-1})$ is $p$ we have $o(\gs(x_i))\mid p\Ra a_{1i}\equiv 0\mod p$ for $2\leq i\leq n$. Since the order of $y_i=(0,\ul{0}^{n-1},0,e^{n-1}_{i-1})$ is $p$ we have $o(\gs(y_i))\mid p\Ra c_{1i}\equiv 0\mod p$ for $2\leq i\leq n$. Similarly for $y_1=(0,\ul{0}^{n-1},1,\ul{0}^{n-1})$
we have $c_{11}\equiv 0\mod p$. 

For $(u_1,\ul{u},w_1,\ul{w})\in G$, let $\widetilde{\ul{u}}=\matcoltwo{\ol{u}_1}{\ul{u}}=(\widetilde{u}_1,\widetilde{u}_2,\ldots,\widetilde{u}_n)^t\in (\ZZ 1)^n,
\widetilde{\ul{w}}=\matcoltwo{w_1}{\ul{w}}=(\widetilde{w}_1,\widetilde{w}_2,\ldots,\widetilde{w}_n)^t\in (\ZZ 1)^n$. Hence we have \equ{\gs(u_1,\ul{u},w_1,\ul{w})=(\widetilde{a}u_1+i_{21}(\widetilde{s}),\gp(A\widetilde{\ul{u}}+C\widetilde{\ul{w}}),D\widetilde{\ul{u}}+B\widetilde{\ul{w}})} for some $\widetilde{a}\in (\ZZ 2),\widetilde{s}\in \ZZ 1$ such that $\widetilde{a}\equiv a_{11}\mod p$.

This computation does not give the explicit form of $\gs$ as we do not know $i_{21}(\widetilde{s})$.
Just similar to the proof of Theorem~\ref{theorem:ExtraSpecialTypeI}(B) we compute $\widetilde{s}$ and obtain
\equ{\widetilde{s}=\widetilde{\ga}(\widetilde{\ul{u}})+\gb(\widetilde{\ul{w}})+\frac 12\langle A\widetilde{\ul{u}},D\widetilde{\ul{u}}\rangle+\frac 12\langle C\widetilde{\ul{w}},B\widetilde{\ul{w}}\rangle+\langle C\widetilde{\ul{w}},D\widetilde{\ul{u}}\rangle} 
for some $\widetilde{\ga},\gb\in ((\ZZ 1)^n)^{\vee}$. Now here we can change $\widetilde{\ga}(\widetilde{\ul{u}})$ to $\ga(\ul{u})$ for some $\ga\in ((\ZZ 1)^{n-1})^{\vee}$ by shifting multiple of $\ol{u}_1$ to the first term in $\widetilde{a}u_1+i_{21}(\widetilde{s})$ to obtain 
$au_1+i_{21}(s)$ without changing the residue class of $\widetilde{a}$ modulo $p$. So we get  
\equan{Sym20}{\gs(u_1,\ul{u},w_1,\ul{w})=(au_1+i_{21}(s),\gp(A\widetilde{\ul{u}}+C\widetilde{\ul{w}}),D\widetilde{\ul{u}}+B\widetilde{\ul{w}})} for some $a\in (\ZZ 2)$ such that $a\equiv a_{11}\mod p$
where 
\equan{Sym21}{s=\ga(\ul{u})+\gb(\widetilde{\ul{w}})+\frac 12\langle A\widetilde{\ul{u}},D\widetilde{\ul{u}}\rangle+\frac 12\langle C\widetilde{\ul{w}},B\widetilde{\ul{w}}\rangle+\langle C\widetilde{\ul{w}},D\widetilde{\ul{u}}\rangle.}

Conversely if $\gs$ is as given in Equation~\ref{Eq:Sym20} and $s$ in Equation~\ref{Eq:Sym21} with the matrix $\ol{\gs}=\mattwo ACDB \in symp^{scalar}(2n,\mbb{F}_p)$ satisfying $\ol{\gs}^t\Gd\ol{\gs}=a_{11}\Gd$ and $a_{12}=\ldots=a_{1n}=c_{11}=c_{12}=\ldots=c_{1n}=0$ then $\gs\in End(G)$. Also in the converse if in addition $a_{11}\not\equiv 0\mod p$, that is, $a\in (\ZZ 2)^{*}$ then $\gs\in Aut(G)$.

The additional consequences of $\gs\in Aut(G)$ are as follows.
We conclude that $\gs$ induces automorphisms of the following three subgroups of $G$. \equ{H=\langle x_1^p, x_2,x_3,\ldots,x_n,y_1,y_2,\ldots,y_n \rangle,K=\mcl{Z}(H)=\langle x_1^p,y_1\rangle,\mcl{Z}(G)=\langle x_1^p\rangle.}
Hence $\gs(y_1)=y_1^{b_{11}}x_1^{pt}$ with $b_{11}\neq 0$, for some $t\in \{0,1,\ldots,p-1\}$ and $b_{j1}=0=c_{j1},2\leq j\leq n$. Now we have $A^tB-D^tC=a_{11}$Id$_{n\times n} \Ra a_{11}b_{11}\equiv a_{11}\mod p \Ra b_{11}=1$. This proves (A),(B),(C),(D).

Now we prove (E). Using Equations~\ref{Eq:TypeII5},~\ref{Eq:TypeII6}, the endomorphic images of any element $g$ in $G$ is given as follows. It is $\{e\}$ if $g=e$. It is $\mcl{Z}(G)$ if $g\in \mcl{Z}(G)\bs \{e\}$. 

\vspace*{0.3cm}
It is $G$ if $g\in G\bs H$ since an element of order $p^2$ can get mapped to any element under an endomorphism. 
First we will show that an element $g=(a,A_{21},d_{11},D_{21})$ $\in G$ of order $p^2$ is automorphic to the element $(1,\ul{0}^{n-1},0,\ul{0}^{n-1})$ where $a\equiv a_{11}\not\equiv 0\mod p$. Consider the automorphism $\gs\in Aut(G)$ such that $\ol{\gs}$ equals \equ{\matfour{a_{11}}{0_{1\times(n-1)}}{0}{0_{1\times(n-1)}}{A_{21}}{a_{11}I_{(n-1)\times(n-1)}}{0_{(n-1)\times 1}}{0_{(n-1)\times(n-1)}}{d_{11}}{D_{12}}{1}{B_{12}}{D_{21}}{0_{(n-1)\times(n-1)}}{0_{(n-1)\times 1}}{I_{(n-1)\times(n-1)}} \text{where }D_{12}=D_{21}^t,B_{12}=\frac{-A_{21}^t}{a_{11}}.}
This automorphism can be used to move $(1,\ul{0}^{n-1},0,\ul{0}^{n-1})$ to $(b,A_{21},d_{11},D_{21})$ where $b\equiv a \equiv a_{11}\mod p$. Now we can change $(b,A_{21},d_{11},D_{21})$ to $(a,A_{21},d_{11},D_{21})$ further by another inner automorphism. Now we will show that \equ{End(G).(1,\ul{0}^{n-1},0,\ul{0}^{n-1})=G.} For this the following matrix can be further used.
\equ{\matfour{0}{0_{1\times(n-1)}}{0}{0_{1\times(n-1)}}{A_{21}}{0_{(n-1)\times(n-1)}}{0_{(n-1)\times 1}}{0_{(n-1)\times(n-1)}}{d_{11}}{0_{1\times(n-1)}}{0}{0_{1\times(n-1)}}{D_{21}}{0_{(n-1)\times(n-1)}}{0_{(n-1)\times 1}}{0_{(n-1)\times (n-1)}}\in symp^{scalar}(2n,\mbb{F}_p).}

\vspace*{0.3cm}

It is $H$ if $g\in H\bs \mcl{Z}(G)$ since a non-central element of order $p$ can get mapped under an endomorphism to any element of order at most $p$. If $g=(pz,\ul{u},w_1,\ul{w})\in H$ then there are two cases. Either $\ul{u}$ or $\ul{w}$ is non-zero or both $\ul{u}$ or $\ul{w}$ are zero and $w_1\neq 0$.

Suppose $\ul{u}$ or $\ul{w}$ is non-zero. Then we show that $g$ is automorphic to $(0,e^{n-1}_1,0,\ul{0}^{n-1})$. Let $M=\mattwo{A_{22}}{C_{22}}{D_{22}}{B_{22}} \in Sp(2n-2,\mbb{F}_p)$ be such that the first column of $M$ is 
$\matcoltwo{\ul{u}}{\ul{w}}$. Now consider an automorphism $\gs\in Aut(G)$ such that $\ol{\gs}$ equals
\equ{\matfour{1}{0_{1\times (n-1)}}{0}{0_{1\times (n-1)}}{A_{21}}{A_{22}}{0_{(n-1)\times 1}}{C_{22}}{d_{11}}{D_{12}}{1}{B_{12}}{D_{21}}{D_{22}}{0_{(n-1)\times 1}}{B_{22}}}
where $D_{12}=D_{21}^tA_{22}-A_{21}^tD_{22}$, $B_{12}=D_{21}^tC_{22}-A_{21}^tB_{22}$.
Here we choose $D_{21}$ and $A_{21}$ such that $(D_{12})_{11} = (D_{21}^tA_{22}-A_{21}^tD_{22})_{11}=w_1$.
Note that such choices of $D_{21}$ and $A_{21}$ exist because the matrix $M$ is invertible and its first column is non-zero. Now $\gs$ moves $(0,e^{n-1}_1,0,\ul{0}^{n-1})$ to $(pz',\ul{u},w_1,\ul{w})\in H$ for some $z'$.
Now using another inner automorphism $(pz',\ul{u},w_1,\ul{w})$ can be mapped to $(pz,\ul{u},w_1,\ul{w})=g$. 
Now we will show that \equ{End(G).(0,e^{n-1}_1,0,\ul{0}^{n-1})=H.}

Now let $M=\mattwo {A_{22}}{0_{(n-1)\times(n-1)}}{D_{22}}{0_{(n-1)\times(n-1)}}\in symp^{scalar}(2n-2,\mbb{F}_p)$ where the first column of $A_{22}$ and $D_{22}$ are given and rest of the columns of $A_{22},D_{22}$ are zero. The following matrix can be further used to show that $End(G).(0,e^{n-1}_1,0,\ul{0}^{n-1})=H$.
\equ{\matfour{0}{0_{1\times(n-1)}}{0}{0_{1\times(n-1)}}{0_{(n-1)\times 1}}{A_{22}}{0_{(n-1)\times 1}}{0_{(n-1)\times(n-1)}}{0}{D_{12}}{0}{0_{1\times (n-1)}}{0_{(n-1)\times 1}}{D_{22}}{0_{(n-1)\times 1}}{0_{(n-1)\times(n-1)}} \text{where }D_{12}=(w,\ul{0}^{n-1})\text{ for given }w.}

Now we consider second case when both $\ul{u}=0=\ul{w}=0$ and $w_1\neq 0$. In this case we show that
\equ{End(G).(pz,\ul{0}^{n-1},w_1,\ul{0}^{n-1})=H.}
For this following matrix can be used.
\equ{\matfour{0}{0_{1\times(n-1)}}{0}{0_{1\times(n-1)}}{0_{(n-1)\times1}}{0_{(n-1)\times(n-1)}}{C_{21}}{0_{(n-1)\times(n-1)}}{0}{0_{1\times(n-1)}}{b_{11}}{0_{1\times(n-1)}}{0_{(n-1)\times 1}}{0_{(n-1)\times(n-1)}}{B_{21}}{0_{(n-1)\times(n-1)}}\in symp^{scalar}(2n,\mbb{F}_p).} 
This proves (E).

Now we prove (F). Using Equations~\ref{Eq:TypeII5},~\ref{Eq:TypeII6}, the automorphism orbits in $G$ are given as follows.
The identity element $\{e\}$ is an orbit. The non-identity central elements $\mcl{Z}(G)\bs \{e\}$ is another orbit. For any automorphism $\gs$ with $\ol{\gs}=\mattwo ACDB$ we have $c_{11}=c_{21}=\ldots=c_{n1}=0$,
$b_{11}=1,b_{21}=b_{31}=\ldots=b_{n1}=0$. So the set $\mcl{O}_b=p(\ZZ 2)\times \{\ul{0}^{n-1}\} \times \{b\}\times \{\ul{0}^{n-1}\}$ for $b\in (\ZZ 1)^{*}$ is an orbit.
We observe that elements of order $p^2$ forms an orbit, that is, $G\bs H$ is an orbit and for $n>1$ the set $H\bs K=H\bs \mcl{Z}(H)$ is an orbit. This proves (F).

Now we prove (G). Any element in $\mcl{O}_{b_1}$ is endomorphic to any element in $\mcl{O}_{b_2}$ for 
$b_1,b_2\in (\ZZ 1)^{*}$. However for $0\neq b_1\neq b_2\neq 0$ any element of $\mcl{O}_{b_1}$ is not automorphic to any element of $\mcl{O}_{b_2}$. This implies that the endomorphism semigroup does not induce a partial order on the automorphism orbits.

This completes the proof of second main Theorem~\ref{theorem:ExtraSpecialTypeII}.
\end{proof}
For $\Gf_2,\widetilde{\Gf}_2$ as defined in Section~\ref{sec:CommDiag} we describe the group Im$(\Gf_2)=$ Im$(\widetilde{\Gf}_2)\subs Sp^{scalar}(2n,\mbb{F}_p)$ and set of endomorphisms in $End(\frac{G}{\mcl{Z}(G)})=M_{2n}(\mbb{F}_p)$ which are induced by the elements in the endomorphism semigroup of $G=ES_2(p,n)$.
\begin{prop}
\label{prop:ImageES2}
Let $G=ES_2(p,n)$. Then
\begin{enumerate}
\item 
 Im$(\Gf_2)=$ Im$\bigg(Aut(G)\lra Aut(\frac{G}{\mcl{Z}(G)})\bigg)=\bigg\{\ol{\gs}=\mattwo ACDB\in Sp^{scalar}(2n,$ $\mbb{F}_p)\mid A=[a_{ij}],B=[b_{ij}],C=[c_{ij}],D=[d_{ij}]\in M_n(\mbb{F}_p)\text{ with }a_{11}\neq 0,b_{11}=1,a_{12}=\ldots=a_{1n}=c_{11}=c_{12}=\ldots=c_{1n}=0=c_{21}=c_{31}=\ldots=c_{n1}=b_{21}=b_{31}=\ldots=b_{n1}\text{ and }\ol{\gs}^t\Gd\ol{\gs}=a_{11}\Gd\bigg\}$.
\item Im$\bigg(End(G)\lra End(\frac{G}{\mcl{Z}(G)})\bigg) = $ Im$(\Gf_2) \bigsqcup \bigg\{\ol{\gs}=\mattwo ACDB\in symp^{scalar}(2n,$ $\mbb{F}_p)\mid A=[a_{ij}],B=[b_{ij}],C=[c_{ij}],D=[d_{ij}]\in M_n(\mbb{F}_p)\text{ with }a_{11}=a_{12}=\ldots=a_{1n}=c_{11}=c_{12}=\ldots=c_{1n}=0 \text{ and }\ol{\gs}^t\Gd\ol{\gs}=0_{2n\times 2n}\bigg\}$.
\item $\gs \in Aut(G)$ is an inner-automorphism if and only if $\ol{\gs}=$Id$_{2n\times 2n}$.
In this case for any $(u_1,\ul{u},w_1,\ul{w}) \in G$ with $\widetilde{\ul{w}}=\matcoltwo{w_1}{\ul{w}}$ we have  \equ{\gs(u_1,\ul{u},w_1,\ul{w})=(au_1+i_{21}\big(\ga(\ul{u})+\gb(\widetilde{\ul{w}})\big),\ul{u},w_1,\ul{w})} for some $\ga \in ((\ZZ 1)^{n-1})^{\vee}, \gb\in ((\ZZ 1)^n)^{\vee},a\in (\ZZ 2)^{*}$ such that $a\equiv 1\mod p$. 
\item We have an exact sequence 
\equ{1\lra \frac{G}{\mcl{Z}(G)}\cong Inn(G)\hookrightarrow Aut(G) \lra \text{Im}(\Gf_2)\lra 1.}
\item \equa{\mid \text{Im}(\Gf_2)\mid =p^{2n-1}(p-1)\mid Sp(2n-2,\mbb{F}_p)\mid.}
\item \equa{\mid Aut(G)\mid&=p^{2n}\mid \text{Im}(\Gf_2)\mid \\
&=p^{n^2+2n}(p-1)\us{j=1}{\os{n-1}{\prod}}(p^{2j}-1).}
\end{enumerate}
\end{prop}
\begin{proof}
This follows from the proof of second main Theorem~\ref{theorem:ExtraSpecialTypeII}.
\end{proof}
The cardinality of $End(G)$ for $G=ES_2(p,n)$ is computed in Section~\ref{sec:Comb}, Theorem~\ref{theorem:CombExtraSpecial}.
\section{\bf{Order of Endomorphism Semigroups of Extra-Special p-Groups}}
\label{sec:Comb}
In this section we compute the cardinality of $End(G)$ for $G=ES_i(p,n),i=1,2$ for an odd prime $p$ and a positive integer $n$. First we note that Im$\bigg(End(G)\lra End(\frac{G}{\mcl{Z}(G)})\bigg)$ is a disjoint union of Im$\bigg(Aut(G)\lra Aut(\frac{G}{\mcl{Z}(G)})\bigg)$ and an algebraic set defined over $\mbb{F}_p$ given as follows.
Let $\langle\langle*,*\rangle\rangle:\mbb{F}_p^{2n}\times \mbb{F}_p^{2n}\lra \mbb{F}_p$ be the non-degenerate symplectic bilinear form given by 
\equ{\langle\langle v,w\rangle\rangle= \us{i=1}{\os{n}{\sum}}(v_iw_{n+i}-v_{n+i}w_i).}
Let $e_i=e_i^{2n},f_i=e^{2n}_{n+i}\in \mbb{F}_p^{2n},1\leq i\leq n$ be the standard basis such that $\langle\langle e_i,f_j\rangle\rangle=\gd_{ij},\langle\langle e_i,e_j\rangle\rangle=0=\langle\langle f_i,f_j\rangle\rangle,1\leq i,j\leq n$. Let $V_1=\langle e_2,\ldots,e_n,f_1,f_2,\ldots,f_n \rangle$.
Let $E_i=$ Im$\bigg(End(G)\lra End(\frac{G}{\mcl{Z}(G)})\bigg)$ where $G=ES_i(p,n),i=1,2$.
Then the following holds.
\begin{itemize}
\item If $G=ES_1(p,n)$ then $E_1=$ Im$(\Gf_1)\bigsqcup X$ where the algebraic set $X=\{N\in M_{2n}(\mbb{F}_p)\mid N^t\Gd N=0\}$ and $\Gf_1$ is as defined in Section~\ref{sec:CommDiag}. So $|End(G)|= p^{2n}|E_1|$ using Equations~\ref{Eq:TypeI2},\ref{Eq:TypeI3} in Theorem~\ref{theorem:ExtraSpecialTypeI}.
\item If $G=ES_2(p,n)$ then $E_2=$ Im$(\Gf_2)\bigsqcup Y$ where the algebraic set $Y=\{N\in M_{2n}(\mbb{F}_p)\mid N^t\Gd N=0,$ Im$(N)\subseteq V_1\}$ and $\Gf_2$ is as defined in Section~\ref{sec:CommDiag}. So $|End(G)|= p^{2n}|E_2|$ using Equations~\ref{Eq:TypeII5},\ref{Eq:TypeII6} in Theorem~\ref{theorem:ExtraSpecialTypeII}.
\end{itemize}

\begin{definition}[Isotropic Subspace]
Let $\langle\langle*,*\rangle\rangle:\mbb{F}_p^{2n}\times \mbb{F}_p^{2n}\lra \mbb{F}_p$ be a non-degenerate symplectic bilinear form. A subspace $W\subs \mbb{F}_p^{2n}$ is said to be {\it isotropic} if for all $v,w\in W,\langle\langle v,w\rangle\rangle=0$. 
\end{definition}
It is well known that the $p$-binomial coefficient $\binom{n}{k}_p$ is a polynomial in $p$ with non-negative integer coefficients for any $0\leq k\leq n$ and $n\neq 0$.
Now we state a theorem about enumeration.
\begin{thm}
\label{theorem:CardAlgSet}
Let $\langle\langle*,*\rangle\rangle:\mbb{F}_p^{2n}\times \mbb{F}_p^{2n}\lra \mbb{F}_p$ be the standard non-degenerate symplectic bilinear form. Let $e_i=e_i^{2n},f_i=e^{2n}_{n+i}\in \mbb{F}_p^{2n},1\leq i\leq n$ and $V_1=\langle e_2,e_3,\ldots,e_n,$ $f_1,f_2,\ldots,f_n\rangle$. Let $X=\{N\in M_{2n}(\mbb{F}_p)\mid N^t\Gd N=0\},Y=\{N\in M_{2n}(\mbb{F}_p)\mid N^t\Gd N=0,$ Im$(N)\subseteq V_1\}$.  For $0\leq k\leq n, Isot_k(\mbb{F}_p^{2n})=\{W\subs \mbb{F}_p^{2n}\mid W \text{ is a }\linebreak k\operatorname{-}\text{dimensional isotropic subspace}\}$ and  $Isot_k(V_1)=\{W\subs V_1\subs \mbb{F}_p^{2n}\mid W \text{ is a }\linebreak k\operatorname{-}\text{dimensional isotropic subspace}\}$. Let $\ga_k(p,n)=|Isot_k(\mbb{F}_p^{2n})|,\gb_k(p,n)=|Isot_k(V_1)|$. Let $\gga_k(p,n)=|\{f:\mbb{F}_p^{2n}\twoheadrightarrow \mbb{F}_p^k\mid f \text{ is a surjective linear map}\}|$. Then we have the following.
\begin{enumerate}
\item $|X|=\us{k=0}{\os{n}{\sum}}\ga_k(p,n)\gga_k(p,n)$.
\item $|Y|=\us{k=0}{\os{n}{\sum}}\gb_k(p,n)\gga_k(p,n)$.
\item For each $0\leq k\leq n, \ga_k(p,n),\gb_k(p,n)$ are polynomials in $p$ with non-negative integer coefficients with 
\begin{enumerate}[label=(\alph*)]
\item $\ga_0(p,n)=1$ and for $1\leq k\leq n,\ga_k(p,n)=\binom{n}{k}_p\ \us{i=0}{\os{k-1}{\prod}}(p^{n-i}+1)$.
\item $\gb_0(p,n)=1,\gb_1(p,n)=\binom{2n-1}{1}_p$ and for $2\leq k\leq n$\\ $\gb_k(p,n)=\bigg(p^k(p^{n-k}+1)\binom{n-1}{k}_p+\binom{n-1}{k-1}_p\bigg)\us{i=1}{\os{k-1}{\prod}}(p^{n-i}+1)$.
\end{enumerate}
\item For each $0\leq k\leq n, \gga_k(p,n)$ is a polynomial in $p$ with integer coefficients with 
$\gga_0(p,n)=1$ and for $1\leq k\leq n, \gga_k(p,n)=p^{2nk}-\us{i=0}{\os{k-1}{\sum}}\binom{k}{i}_p\ \gga_i(p,n)$.
\end{enumerate} 
\end{thm}
\begin{proof}
If $N\in M_{2n}(\mbb{F}_p)$ and $N^t\Gd N=0$, that is, Im$(N)$ is an isotropic subspace of $\mbb{F}_p^{2n}$ then $\dim($Im$(N))\leq n$. So $(1)$ and $(2)$ immediately follow. 

Now we prove $3(a)$. It is clear that $\ga_0(p,n)=1$.
For $1\leq k\leq n$, let $T_k=\{(v_1,v_2,\ldots,v_k)\in (\mbb{F}_p^{2n})^{k}\mid (v_1,v_2,\ldots,v_k)$ is an ordered $k$-tuple of linearly independent vectors whose span is isotropic$\}$. Then we have 
\equ{|T_k|=(p^{2n}-1)(p^{2n-1}-p)\ldots(p^{2n-(k-1)}-p^{k-1}).}
Hence we have \equ{\ga_k(p,n)=\frac{(p^{2n}-1)(p^{2n-1}-p)\ldots(p^{2n-(k-1)}-p^{k-1})}{(p^k-1)(p^k-p)\ldots (p^k-p^{k-1})}=\binom{n}{k}_p\ \us{i=0}{\os{k-1}{\prod}}(p^{n-i}+1).} 

Now we prove $3(b)$. It is clear that $\gb_0(p,n)=1,\gb_1(p,n)=\binom{2n-1}{1}_p$. For $2\leq k\leq n$, let $S_k=\{(v_1,v_2,\ldots,v_k)\in (\mbb{F}_p^{2n})^{k}\mid (v_1,v_2,\ldots,v_k)$ is an ordered $k$-tuple of linearly independent vectors whose span is isotropic and is contained in $V_1\}$.

Let $L\subs (\mbb{F}_p^{2n},\langle\langle *,*\rangle\rangle)$ be a subspace. We make the following observations.
\begin{itemize}
\item $\dim L+\dim L^{\perp}=2n,(L^{\perp})^{\perp}=L, V_1^{\perp}=\langle f_1\rangle$.
\item $f_1\in L \Llra V_1^{\perp}\subseteq L \Llra L^{\perp} \subseteq V_1 \Llra L^{\perp} \cap V_1=L^{\perp}$.
\item $f_1\nin L \Llra V_1^{\perp}\nsubseteq L \Llra L^{\perp} \nsubseteq V_1 \Llra L^{\perp} \cap V_1\subsetneq L^{\perp}$ and of co-dimension one.
\end{itemize}
Let $k=2$. We have $p^{2n-1}-1$ choices for $v_1\in V_1$ out of which $(p-1)$ choices of $v_1$ are non-zero multiples of $f_1$ and $p^{2n-1}-p$ choices of $v_1$ are not multiples of $f_1$. In the first case $v_2\in \langle v_1\rangle^{\perp} \cap V_1$ has $p^{2n-1}-p$ choices. In the latter case there are $p^{2n-2}-p$ choices for $v_2\in \langle v_1\rangle^{\perp} \cap V_1$. So 
\equ{|S_2|=(p-1)(p^{2n-1}-p)+(p^{2n-1}-p)(p^{2n-2}-p)=(p^{2n-1}-p)(p^{2n-2}-1).}
So \equa{\gb_2(p,n)&=\frac{(p^{2n-2}-1)(p^{2n-2}-1)}{(p^2-1)(p-1)}\\&=(p^{n-1}+1)\bigg(p^2(p^{n-2}+1)\binom{n-1}{2}_p+\binom{n-1}{1}_p\bigg).}    
Extending the same argument for $3\leq k\leq n$ we get \equ{|S_k|=(p^{2n-1}-p)(p^{2n-2}-p^2)\ldots(p^{2n-(k-1)}-p^{k-1})(p^{2n-k}-1).}
We also have 
\equa{\gb_k(p,n)&=\frac{(p^{2n-1}-p)(p^{2n-2}-p^2)\ldots(p^{2n-(k-1)}-p^{k-1})(p^{2n-k}-1)}{(p^k-1)(p^k-p)\ldots(p^k-p^{k-1})}\\&=\bigg(p^k(p^{n-k}+1)\binom{n-1}{k}_p+\binom{n-1}{k-1}_p\bigg)\us{i=1}{\os{k-1}{\prod}}(p^{n-i}+1).}
Now we prove $(4)$. It is clear that $\gga_0(p,n)=1$. To compute the number of surjective maps we consider all maps from $\mbb{F}_P^{2n}\lra \mbb{F}_p^k$ and subtract the number of maps of rank less than $k$.
Hence we get for $1\leq k\leq n$, \equ{ \gga_k(p,n)=p^{2nk}-\us{i=0}{\os{k-1}{\sum}}\binom{k}{i}_p\ \gga_i(p,n).}
This completes the proof of the theorem.
\end{proof}
\begin{thm}
\label{theorem:CombExtraSpecial}
\begin{enumerate}
\item For $G=ES_1(p,n)$ we have \equ{|End(G)|=p^{n^2+2n}(p-1)\us{j=1}{\os{n}{\prod}}(p^{2j}-1)+p^{2n}\us{k=0}{\os{n}{\sum}}\ga_k(p,n)\gga_k(p,n).}
\item For $G=ES_2(p,n)$ we have \equ{|End(G)|=p^{n^2+2n}(p-1)\us{j=1}{\os{n-1}{\prod}}(p^{2j}-1)+p^{2n}\us{k=0}{\os{n}{\sum}}\gb_k(p,n)\gga_k(p,n).}
\end{enumerate}
\end{thm}
\begin{proof}
First we observe that for $G=ES_1(p,n), |End(G)|=|Aut(G)|+p^{2n}|X|$ and for $G=ES_2(p,n), |End(G)|=|Aut(G)|+p^{2n}|Y|$ where $X,Y$ are as defined in Theorem~\ref{theorem:CardAlgSet}. Now
using Theorem~\ref{theorem:CardAlgSet}, Corollary~\ref{cor:ImageES1}(3), we conclude $(1)$ and then again using Theorem~\ref{theorem:CardAlgSet} and Proposition~\ref{prop:ImageES2}(6), we conclude $(2)$. This completes the proof of the theorem.
\end{proof}	
\begin{example}
For $n=1$ and $G=ES_1(p,1)$ we obtain $|Aut(G)|=p^3(p-1)(p^2-1)$ and $|End(G)|=p^3(p-1)(p^2-1)+p^2(1+\binom{1}{1}_p(p+1)(p^2-1))=p^6$.

For $n=1$ and $G=ES_2(p,1)$ we obtain $|Aut(G)|=p^3(p-1)$ and $|End(G)|=p^3(p-1)+p^2(1+\binom{1}{1}_p(p^2-1))=2p^4-p^3$.
\end{example}
\section{\bf{An Open Question}}
This article leads to an open question which we pose in this section. In general for a finite group, 
its center and commutator subgroup are characteristic subgroups. However it is not true that an endomorphism maps the center into itself, but an endomorphism maps commutator subgroup into itself. 
Any automorphism or any endomorphism gives rise to a pair of automorphisms and endomorphisms of the commutator subgroup and the abelianization of whole group respectively. The automorphism group and the endomorphism algebra for finite abelian groups are known.
Now we pose the following open question. 
\begin{ques}
\label{ques:ExtraOpen}
Let $p$ be a prime. Let $G$ be a $p$-group such that $G'=[G,G]$ is a non-trivial abelian group, that is, $G$ is a non-abelian metabelian $p$-group. Then: 
\begin{itemize}
\item Determine the automorphism orbits in $G$.
\item Determine the endomorphism semigroup image of any element in $G$.
\item Determine for which type of such groups $G$ the endomorphism semigroup induces a partial order on the automorphism orbits.
\end{itemize}
\end{ques}

Now in addition for the group $G$ in Question~\ref{ques:ExtraOpen}, if the center coincides with the commutator subgroup then any endomorphism maps the center into itself. Moreover for such a group, if $\mcl{Z}(G)$ is elementary abelian, then we have a non-degenerate skew symmetric bilinear map $\frac{G}{\mcl{Z}(G)}\times \frac{G}{\mcl{Z}(G)} \lra \mcl{Z}(G)$. An example of such a group is given below.

\begin{example}
An example of a non-abelian metabelian $p$-group $G$ which satisfies $[G,G]=G'=\mcl{Z}(G)$ and $\mcl{Z}(G)$ is elementary abelian is the Heisenberg group $H^n(\mbb{F}_q)=\mbb{F}_q^n\oplus \mbb{F}_q^n\oplus \mbb{F}_q$ over the field $\mbb{F}_q$ of order $q^{2n+1}$ where $q=p^r$ for some prime $p$. The group structure is defined in a similar manner as in $ES_1(p,n)$. The answer to Question~\ref{ques:ExtraOpen} can be explored in the case of $H^n(\mbb{F}_q)$. 
\end{example}

\textbf{Acknowledgements:}
The work is done while both the authors are post doctoral fellows at Harish-Chandra Research Institute, Allahabad-INDIA. Both the authors thank Prof. Amritanshu Prasad and Prof. Sunil Kumar Prajapati for mentioning the problem of finding automorphism orbits in extra-special p-groups. The authors also thank Prof. Manoj Kumar Yadav for suggesting a lot of improvements in the article.


\begin{thebibliography}{1}
\bibitem{MR1157256} D.~J.~Benson, J.~F.~Carlson, {\it The cohomology of extraspecial groups}, 
Bull. London Math. Soc., Vol. 24, (1992), No. {\bf 3}, pp. 209–235, \url{ https://doi.org/10.1112/blms/24.3.209}, MR1157256

\bibitem{MR1233415} D.~J.~Benson, J.~F.~Carlson, {\it Corrigendum: ‘the Cohomology of Extraspecial groups’}, Bull. London Math. Soc., Vol. 25, (1993), No. {\bf 5}, pp. 498-498, \url{https://doi.org/10.1112/blms/25.5.498}, MR1233415

\bibitem{MR687893} E.~.A.~Bertram, {\it Some applications of graph theory to finite groups}, 
Discrete Math., Vol. 44, No. {\bf 1}, (1983), pp. 31-43, \url{https://doi.org/10.1016/0012-365X(83)90004-3}, MR687893 

\bibitem{MR2126728} A.~Y.~M.~Chin, {\it On non-commuting sets in an extraspecial $p$-group},
J. Group Theory, Vol. 8, Issue 2, (2005), pp. 189–194, \url{https://doi.org/10.1515/jgth.2005.8.2.189}, MR2126728

\bibitem{MR0347959} L.~Dornhoff, {\it Group Representation Theory: Part A, Ordinary Representation Theory},
Pure and Applied Mathematics Series, Vol. 7, M. Dekker Inc, New York, (1971), 254 pages, MR0347959 

\bibitem{MR2793603}	K.~Dutta, A.~Prasad, {\it Degenerations and orbits in finite abelian groups}, 
J. Combin. Theory Ser. A {\bf 118} (2011), no. {\bf 6}, 1685–1694, \url{https://doi.org/10.1016/j.jcta.2011.02.002}, MR2793603

\bibitem{MR0231903} D.~E.~Gorenstein, {\it Finite Groups}, AMS Chelsea Publishing, Vol. 301, (1968), 519 pages, ISBN-13 {\bf 978-0-8218-4342-0}, \url{https://bookstore.ams.org/chel-301/}, MR0231903 

\bibitem{MR0476878} R.~L.~Griess Jr., {\it Automorphisms of Extraspecial Groups and Nonvanishing of Degree $2$ Cohomology}, Pacific J. Math, Vol. 48, No. {\bf 2}, (1973), pp. 403-422, \url{https://projecteuclid.org/euclid.pjm/1102945424}, MR0476878

\bibitem{MR2606849} H.~Liu, Y.~Wang, {\it The automorphism group of a generalized extraspecial $p$-group}, Sci. China Math., Vol. 53, No. {\bf 2}, (2010), pp.315-334, \url{https://doi.org/10.1007/s11425-009-0151-2}, MR2606849

\bibitem{MR3058246} H.~Liu, Y.~Wang, {\it On non-commuting sets in a generalized extraspecial $p$-group},
Acta Math. Sinica, Vol. 55, No. {\bf 6}, (2012), pp. 975-980, (in chinese), MR3058246 

\bibitem{MR3403691} H.~Liu, Y.~Wang, {\it On Non-commuting Sets in Certain Finite $p$-Groups},
Algebra Colloquium, Vol. 22, No. {\bf 4}, (2015), pp. 555-560, \url{https://doi.org/10.1142/S1005386715000474}, MR3403691

\bibitem{MR0486098} H.~Opolka, {\it Projective Representations of Extra-Special p-Groups}, Glasgow Mathematical Journal, Vol. 19, Issue 2 July 1978, pp. 149-152, \url{https://doi.org/10.1017/S0017089500003542}, MR0486098

\bibitem{MR648604} D.~J.~S.~Robinson, {\it A course in the theory of groups}, Graduate Texts in Mathematics, 80, Springer-Verlag, New York-Berlin, 1982, xvii+481pp. ISBN-13 {\bf 978-0-387-94461-6}, \url{https://doi.org/10.1007/978-1-4419-8594-1}, MR648604

\bibitem{MR0297859} D.~L.~Winter, {\it The automorphism group of an extraspecial p-group}, 
Rocky Mountain J. Math, Vol. 2, No. {\bf 2} (SPRING 1972), pp. 159-168, doi: 10.1216/RMJ-1972-2-2-159, \url{https://www.jstor.org/stable/44236249}, \url{ https://projecteuclid.org/euclid.rmjm/1250187219}, MR0297859
\end{thebibliography}
\end{document}